\documentclass[12pt]{article}
\usepackage{latexsym}
\usepackage{amssymb}
\usepackage{amsmath}
\usepackage{amsthm}
\usepackage{rotating}
\usepackage{multirow}
\usepackage{mathrsfs}
\usepackage{wasysym}
\usepackage{slashed}
\usepackage{esint}
\usepackage{nicefrac}
\usepackage{rawfonts}
\input{prepictex}
\input{pictex}
\input{postpictex}
\usepackage[format=hang,font=sf,labelfont=bf]{caption}
\usepackage{pifont}
\usepackage[OT2,OT1]{fontenc}
\def\cyr{%
\renewcommand\rmdefault{wncyr}%
\renewcommand\sfdefault{wncyss}%
\renewcommand\encodingdefault{OT2}%
\normalfont
\selectfont}
\DeclareMathAlphabet{\zap}{OT1}{pzc}{m}{it}
\DeclareTextFontCommand{\textcyr}{\cyr}

\newcommand{\bd}{\text{\cyr  b}}
\newcommand{\pip}{\text{\cyr  f}}
\newcommand{\proj}{\text{\cyr  p}}
\newcommand{\incl}{\text{\cyr  i}}
\def\CC{\mathbb C}

\def\dir{\slashed{D}}
\newtheorem{main}{Theorem}

\DeclareMathOperator{\Aut}{Aut}

\newtheorem{thm}{Theorem}
\newtheorem*{them}{Theorem}
\newtheorem{prop}{Proposition}
\newtheorem{lem}[main]{Lemma}
\newtheorem{lemma}{Lemma}
\newtheorem{cor}{Corollary}
\newtheorem{defn}{Definition}

\def\ZZ{{\mathbb Z}}
\def\RR{{\mathbb R}}
\def\CP{{\mathbb C \mathbb P}}
\def\dir{\not \negthickspace D}

\DeclareMathOperator{\kod}{Kod}

\def\barroman#1{\sbox0{#1}\dimen0=\dimexpr\wd0+1pt\relax
  \makebox[\dimen0]{\rlap{\vrule width\dimen0 height 0.06ex depth 0.06ex}%
    \rlap{\vrule width\dimen0 height\dimexpr\ht0+0.03ex\relax 
            depth\dimexpr-\ht0+0.09ex\relax}%
    \kern.5pt#1\kern.5pt}}
\begin{document}

\title{Kodaira Dimension \&  the Yamabe Problem, II}

\author{Michael Albanese and Claude LeBrun\thanks{Supported 
in part by  NSF grant DMS-1906267.}} 
  
\date{June 27, 2021; revised December 7, 2021}

\maketitle

\begin{abstract}  For compact complex surfaces $(M^4, J)$ 
of K\"ahler type, it was previously shown \cite{lky}  that  the sign of the Yamabe invariant $\mathscr{Y}(M)$ 
only depends on  the  Kodaira dimension $\kod (M, J)$. In this paper,  we  prove that this pattern in fact extends to
 {all} compact complex surfaces {\sf except} those of 
class {\sf VII}. In the process, we also reprove  a  result from \cite{alba} that explains why  the exclusion of  
class  {\sf VII} is essential here. 
  \end{abstract}

The {\em Yamabe invariant}  $\mathscr{Y}(M)$ of a smooth compact $n$-manifold $M$ is a diffeomorphism invariant, 
introduced  by \cite{okob} and \cite{sch},  
that largely captures the global behavior of  the scalar curvature $s=s_g$  for all possible Riemannian metrics $g$ on $M$. 
A cartoon version of   the idea would be to  just take   the supremum
of the scalar curvatures of all unit-volume
 constant-scalar-curvature
metrics on $M$.  However, in the precise definition of the invariant, one 
 only considers 
 those constant-scalar-curvature metrics that are
{\sf Yamabe minimizers}; this   does not change the
{\sf sign} of the supremum, but it does in particular 
guarantee that it  is always finite. Thus   the {\sf Yamabe invariant} of $M$ is    defined  to be 
\begin{equation}
\label{yamdef}
\mathscr{Y}(M) := \sup_\gamma \inf_{g\in \gamma} \frac{\int_M s_g~d\mu_g}{[ \int_Md\mu_g]^{\frac{n-2}{n}}} 
\end{equation}
where  $d\mu_g$ denotes the metric volume measure and 
$\gamma$ varies over all conformal classes of Riemannian metrics $g$ on $M$. With this
convention, $\mathscr{Y}(M) > 0$ iff $M$ admits a metric  of scalar curvature $s> 0$,
whereas $\mathscr{Y}(M) \geq  0$ iff  $M$ admits a unit-volume metric  of  scalar curvature $s >  -\epsilon$
for  every $\epsilon > 0$.

When $n=4$ and $M$ is the underlying smooth manifold of a compact complex surface $(M^4,J)$, 
our purpose here is to study the interplay between the 
Yamabe invariant and a  complex-analytic invariant of $(M,J)$ called the {\em Kodaira dimension}. 
For a compact complex $m$-manifold $(M^{2m},J)$, the Kodaira dimension  \cite{bpv,GH,lazar1} is by definition 
$$\kod (M, J) = \limsup_{j \to +\infty}  \frac{\log \dim H^0 (M,\mathcal{O}(K^{\otimes j}))}{\log j}$$
where $K=\Lambda^{m,0}$ is the {\sl canonical line bundle} of $(M,J)$. 
The value of this  invariant  always belongs to $\{ -\infty, 0, 1, \ldots, m\}$, because 
the Kodaira dimension actually coincides with the 
 largest complex dimension of the image of $M\dasharrow \mathbb{P} [ H^0 (M,\mathcal{O}(K^{\otimes j}))^*]$ among all the possible   ``pluricanonical''  maps 
associated with the  line bundles $K^{\otimes j}$, $j \in \ZZ^+$, subject to   the rather  unusual convention of 
setting  $\dim \varnothing :=-\infty$.  

In complex dimensions $m\geq 3$, the Kodaira dimension is not a diffeomorphism invariant \cite{cat,lebyam,rares}, and hence  is essentially unrelated to 
 the Yamabe invariant. By contrast, the situation is completely  different when $m=1$ or $2$. When $m=1$, classical Gauss-Bonnet immediately implies that  \eqref{yamdef} simplifies to become $\mathscr{Y}(M^2) = 4\pi \chi (M^2)$, while 
classical Riemann-Roch  implies that  the  Kodaira dimension is determined by  the sign of the 
 Euler characteristic.  
When $m=2$, by combining 
 results from Seiberg-Witten theory
  with  constructions  of special   sequences of 
  Riemannian metrics, the second author was able to prove   the following  analogous result  \cite{lky}:

\begin{them}[LeBrun] 
Let $M$ be the  smooth 4-manifold   underlying  a
 compact complex  surface $(M^4,J)$ of {K\"ahler type}.    
Then 
\begin{eqnarray*}
\mathscr{Y}(M) > 0 &\Longleftrightarrow& \kod (M,J) = -\infty ,\\
\mathscr{Y}(M) = 0 &\Longleftrightarrow& \kod (M,J) = 0 \mbox{ or } 1 ,\\
  \mathscr{Y}(M)<  0 &\Longleftrightarrow& \kod (M,J) = 2.
  \end{eqnarray*}
\end{them}

 Here the K\"ahler-type condition is equivalent \cite{bpv,nick,siu} to requiring 
that $b_1(M)$ be even. In the present paper, we will prove a partial  generalization
of this  result that allows for the possibility that  $b_1(M)$  might be odd:

\begin{main} 
\label{ellipsis}
Let  $M$ be the underlying smooth   $4$-manifold of
any  compact complex surface $(M^4,J)$  of Kodaira dimension $\neq -\infty$. Then 
\begin{eqnarray*}
\mathscr{Y}(M) = 0 &\Longleftrightarrow& \kod (M,J) = 0 \mbox{ or } 1, \\
  \mathscr{Y}(M)<  0 &\Longleftrightarrow& \kod (M,J) = 2.
  \end{eqnarray*}
\end{main}

One  cornerstone of Kodaira's classification of complex surfaces  \cite{bpv,GH} is the {\em blow-up} operation, 
which replaces a point of a complex surface $Y$ with a $\CP_1$ of normal bundle $\mathcal{O}(-1)$; this then produces 
a new complex surface $M$ that is diffeomorphic to  $Y \# \overline{\CP}_2$, where 
$ \overline{\CP}_2$ denotes the smooth oriented $4$-manifold obtained by equipping  $\CP_2$ with the non-standard orientation. Conversely, any complex surface $M$ containing a $\CP_1$ of normal bundle $\mathcal{O}(-1)$
can be {\em blown down} to produce a new complex surface $Y$ of which $M$ then becomes the  blow-up. 
In principle, this blow-down procedure can  then  be iterated, but the process must terminate after finitely many steps, 
because each blow-down decreases $b_2$ by $1$. When  a complex surface
$X$ cannot be blown down, it is called {\em minimal}, and the upshot is that any
complex surface $M$ can be obtained from a minimal complex surface 
$X$ by blowing up finitely many times. In this situation,  one then says that $X$ is a {\em minimal model} of $M$. 
Blowing up or down always  leaves the Kodaira dimension unchanged. Moreover, 
 the minimal model of a complex surface is actually 
{\em unique} whenever  $\kod \neq -\infty$.

Our proof of Theorem \ref{ellipsis} also yields the 
 the following additional main result, which was previously proved
in \cite{lno,lky} when $b_1(M)$ is even: 

\begin{main} 
\label{parity}
Let  $(M,J)$ be a 
 compact complex  surface with $\kod \neq -\infty$, and let 
 $(X, J^\prime)$ be its minimal model.
 Then 
 $$\mathscr{Y}(M) = \mathscr{Y}(X).$$
\end{main}

Given the results previously proved in \cite{lky},  Theorems \ref{ellipsis} and \ref{parity} just require one to 
show that no properly elliptic complex surface with $b_1$ odd can admit  Riemannian metrics
of positive scalar curvature. In \S \ref{reduction} below, we give a  covering argument that reduces this claim to a lemma
asserting that blow-ups of $T^2$-bundles over high-genus Riemann surfaces cannot admit 
positive-scalar-curvature Riemannian metrics. 
The subsequent  two sections then give two entirely different proofs of  this lemma.
Our first proof, detailed in \S \ref{schoen-yau} and previously sketched in \cite{glblb}, uses  the stable-minimal-hypersurface method of  Schoen and Yau \cite{syrger}. Our second proof, laid out in \S \ref{monopoly} below, 
 instead deduces the lemma from a curvature estimate implicitly proved by  Kronheimer \cite{K},  and is 
 closer in spirit to  \cite{lno,lky} because of the leading role played by 
  the Seiberg-Witten equations.  We then go on, in \S \ref{pathology}, to explain why Theorems \ref{ellipsis}
 and \ref{parity} require  the exclusion of the $\kod = -\infty$ case, while in the process giving a simplied
 proof of the main result of \cite{alba}. We then conclude, in \S \ref{conclusion}, with a discussion  of  pertinent related  results
 and various open problems. 

\pagebreak

\section{Reducing the Problem to a Lemma}
\label{reduction}

Any complex surface of $\kod = 2$ is algebraic \cite{bpv}, and hence of K\"ahler type. Since the K\"ahler-type 
versions of Theorems \ref{ellipsis}
and \ref{parity}
were previously proved in \cite[Theorems A and 2]{lky},  we therefore only need to address the  cases
of  complex surfaces $(M^4,J)$ with $b_1$ odd and  $\kod =0$ or $1$. Any such  complex surface is necessarily 
elliptic \cite{bpv}, in the sense that it must   admit a holomorphic map to a complex curve with generic
fiber diffeomorphic to  $\mathbb{T}^2$. It therefore follows \cite[Corollary 1]{lky} that
any such $(M,J)$ satisfies $\mathscr{Y}(M) \geq 0$, because one can  construct   sequences of Riemannian
metrics on such spaces with $\int s^2d\mu \to 0$. However, any complex surface 
with $b_1$ odd and
$\kod =0$ 
 is finitely covered by some blow-up of a primary Kodaira surface, which is then a  symplectic $4$-manifold with 
$b_+=2$; since a celebrated  result of Taubes \cite{taubes} then implies that this finite cover 
 carries a Seiberg-Witten basic class, and therefore  does not admit metrics of positive scalar curvature, 
one may therefore conclude \cite[p. 153]{lky} 
that    $\mathscr{Y}(M) = 0$. To prove Theorems \ref{ellipsis} and \ref{parity}, it therefore suffices to prove
that if  $\kod (M^4,J) =1$ and $b_1(M)\equiv 1\bmod 2$, then $M$ cannot admit a  Riemannian metric of
positive scalar curvature. We will deduce this from the following narrower statement:

\begin{lem}
\label{crux} 
Let $\Sigma$ denote a compact Riemann surface of genus $\geq 2$,  and let $N\to \Sigma$ be a principal $\mathbf{U}(1)$-bundle of non-zero 
Chern class. Set $X= N\times S^1$, and let $M= X\# k \overline{\CP}_2$ for some integer $k\geq 0$. Then 
$M$ does not admit any Riemannian metric $g$ of positive scalar curvature. 
\end{lem}

In  \S\S \ref{schoen-yau}--\ref{monopoly} below, 
we will  prove this lemma twice, in  two entirely  different ways. 
In the meantime, though, we   begin by carefully explaining why this lemma suffices to imply our main results. 

\begin{prop}
\label{lynx} 
 Lemma \ref{crux} implies Theorems \ref{ellipsis} and \ref{parity}. 
\end{prop} 
\begin{proof} Per the above discussion,  it suffices to show  that  whenever $b_1(M)$  is {odd} and  $\kod (M,J) =1$, 
the  smooth 
$4$-manifold $M$ cannot admit   metrics of positive scalar curvature. 
Let $(X,J^\prime)$  denote the minimal model of $(M,J)$. 
 Because $\kod (X,J^\prime) = \kod (M,J) =1$,  normalization of some pluricanonical map 
   defines  
 an elliptic fibration  $\varpi : X\to \Sigma$, where  $\Sigma$ is a   smooth  connected complex curve.  Because $b_1(X)=b_1(M)$ is odd,
 there must be an element of $H^1(X, \mathcal{O})$ that is non-trivial on some fiber. An argument due to 
Br\^{\i}nz\u{a}nescu \cite{brin-elliptic} therefore shows that first direct image sheaf 
 $\varpi_{*}^1 \mathcal{O}$ must be   a holomorphically trivial  line bundle on $\Sigma$, 
 because its degree is {\em a priori}   non-positive by  \cite[Theorem III.18.2]{bpv}. Hence 
no  fiber 
can  just be a union of rational curves,   and  $\varpi$ can therefore at worst  have multiple fibers with smooth reduction. 
 We 
 now equip $\Sigma$ with an orbifold structure by giving each point 
 a  weight equal to the multiplicity of the corresponding fiber. Because $\kod (X,J^\prime) = \kod (M,J) = 1$, 
 we  must   have $\chi^{orb}(\Sigma) \leq 0$,  because $(X,J^\prime)$ would  otherwise \cite[\S 2.7]{FMbook} be a Hopf surface, and so 
 have $\kod = -\infty$. In particular,  $\Sigma$ must be a    {\em good orbifold}  in the sense of Thurston \cite{thurston-orb}.
 Thus   $\Sigma= \widehat{\Sigma}/\Gamma$, where $\widehat{\Sigma}$ is a smooth complex curve of positive genus, 
and where the finite group $\Gamma$ acts  biholomorphically on $\widehat{\Sigma}$. 
Pulling 
$X$ back to $\widehat{\Sigma}$ then produces an (unramifed)  cover $\widehat{X}\to X$ equipped with a 
holomorphic submersion  $\widehat{\varpi}: \widehat{X}\to \widehat{\Sigma}$.  However, the $j$-invariant of the fibers now defines a holomorphic map 
$\widehat{\Sigma}\to \CC$, and this map is of course constant because $\widehat{\Sigma}$ is compact. Thus 
$\widehat{\varpi}$ is locally {holomorphically}  trivial, with fibers isomorphic to  some elliptic curve $E$, 
and the obstruction to this being a principal $E$-bundle is then measured by 
 the rotational  monodromy map 
$\pi_1 (\widehat{\Sigma})\to \Aut (E)/E$. However,  $\Aut (E)$ is compact, so 
$\Aut (E)/E = \pi_0 (\Aut(E))$ is finite, and 
we may therefore kill  this monodromy simply  by replacing $\widehat{\Sigma}$ with a finite  cover if necessary. 
This shows \cite{brin-elliptic,wall-elliptic} that there is a finite cover $\widehat{X}$ of $X$ that  is just 
a principal $[\mathbf{U}(1)\times \mathbf{U}(1)]$-bundle over a compact complex curve $\widehat{\Sigma}$.
Since $\kod (\widehat{X}) = 1$,
the base $\widehat{\Sigma}$ must necessarily have genus $\geq 2$. Moreover,  $b_1(\widehat{X})$ is necessarily odd, because 
 otherwise we could produce a forbidden K\"ahler metric on $X$ by taking fiber-averages of local push-forwards of some K\"ahler metric
on $\widehat{X}$. It follows that the Chern classes  of the two $\mathbf{U}(1)$-factors of our principal $[\mathbf{U}(1)\times \mathbf{U}(1)]$-bundle $\widehat{X}\to \widehat{\Sigma}$, must be 
 linearly {\em dependent} over $\mathbb{Q}$. 
Thus,   by again passing to a cover  and then changing basis if necessary, we then may arrange for  
exactly one of these Chern classes to be  non-zero. This shows that $X$ has an unbranched cover 
  $\widehat{X} \approx N\times S^1$, where  $N\to \widehat{\Sigma}$ is a principal $\mathbf{U}(1)$-bundle of non-zero degree 
over a Riemann surface of genus $\geq 2$. 
Since $M$  is obtained from $X$ by blowing up points, we 
therefore   obtain an induced  unbranched cover $\widehat{M} \to M$ with $\widehat{M}\approx (N\times S^1) \# k \overline{\CP}_2$.
Lemma \ref{crux} now asserts that $\widehat{M}$ cannot carry a metric of positive scalar curvature, and, since we can  pull Riemannian metrics back via 
  $\widehat{M}\to M$,   this then implies that  $M$ cannot carry a positive-scalar-curvature 
    metric either. Hence Lemma \ref{crux} implies that $\mathscr{Y}(M)=0$. Moreover, since 
the same argument in particular applies to $X$, Lemma \ref{crux}  also  implies $\mathscr{Y}(M)=\mathscr{Y}(X)$. 
Hence  Lemma \ref{crux}  implies both Theorem \ref{ellipsis} and 
Theorem \ref{parity},    as claimed. 
 \end{proof}

 \pagebreak

\section{Proof via Stable Minimal Hypersurfaces}
\label{schoen-yau}

In this section, we give a proof of Lemma \ref{crux} largely    based on 
the Schoen-Yau  stable-minimal-hypersurface method \cite{syrger}. For the broader context, see \S \ref{conclusion}.

Let   $(M,g)$ be a smooth compact
oriented Riemannian $(\ell+1)$-manifold, $\ell\leq 6$, and  let $\mathsf{a} \in H^1 (M, \ZZ)$ be a non-trivial 
cohomology class. 
 Compactness results in geometric measure theory \cite{federer} 
 guarantee that there is a mass-minimizing  integral current $\mathscr{Z}$ that
represents the Poincar\'e dual homology class in $H_{\ell}(M, \ZZ )$, and, because $\ell < 7$, 
regularity results \cite{simons} then guarantee  \cite{lawson-montreal} that this current is just a sum 
$$\mathscr{Z}=\sum n_j Z_j$$ 
of  disjoint smooth embedded 
oriented compact connected 
$\ell$-dimensional hypersurfaces  $Z_j\subset M$ with positive integer multiplicities.

Because $\mathscr{Z}$ is mass-minimizing, each $Z_j$ must be a stable minimal hypersurface. 
Thus, if $Z= Z_j$ for some $j$, and if 
$Z(t)$ is any  smooth family of hypersurfaces in $M$ with $Z(0)= Z$,
then the $\ell$-dimensional volume $\mathscr{A}(t)$ of $Z(t)$ must satisfy
$$
\mathscr{A}^{\prime} (0) = 0, \qquad \mathscr{A}^{\prime \prime} (0) \geq  0.
$$
If $h$ and $\gemini$ denote the induced metric and second fundamental form of $Z$, the vanishing of $\mathscr{A}^{\prime} (0)$ 
is of course equivalant to the vanishing of the mean curvature $H = h^{ij} \gemini_{ij}$. On the other hand, 
if the normal component of  $Z^{\prime}(0)$ for this family is given by 
$u \, \vec{\mathsf{n}}$, where $\vec{\mathsf{n}}$ is the unit normal vector of $Z$,
then, in conjunction with  the Gau{ss}-Codazzi equations,
  Jim  Simons' second-variation formula \cite[Theorem 3.2.2]{simons}  tells us that 
$$
\mathscr{A}^{\prime \prime} (0) =\int_Z \left[|\nabla u|^2 + 
 \frac{1}{2}(s_h-s_g -  |\gemini|^2)u^2\right] \,  d\mu_h
$$
so that the  homologically-mass-minimizing property  of $\mathscr{Z}$ implies that 
\begin{equation}
\label{jimmy}
 \int_Z \left[2|\nabla u|^2 + 
s_h u^2\right]\, d\mu_h \geq  \int_Z s_g u^2\, d\mu_h
\end{equation}
for every smooth function $u$ on $Z$. 

We  now assume  that  $(M,g)$ has positive scalar curvature $s_g > 0$,
and then notice that \eqref{jimmy} implies that 
\begin{equation}
\label{james}
 \int_Z \left[2|\nabla u|^2 + 
s_h u^2\right]\, d\mu_h  > 0
\end{equation}
for every smooth function $u\not\equiv 0$ on $Z$. 
If $\ell =2$, plugging  $u \equiv 1$ into \eqref{james} immediately  implies  that $\chi (Z) > 0$ by classical Gauss-Bonnet, and 
classical uniformization therefore tells us that there is a conformal diffeomorphism between $(Z^2,h)$ and 
and $(S^2, \widehat{h})$, where $\widehat{h}$ is the  standard  unit-sphere metric of scalar curvature $+2$.  Similarly, when $\ell \geq 3$,
the induced metric $h$ is also  conformal to a metric $\widehat{h}$ of positive scalar curvature. Indeed, 
if   $p = \frac{2\ell}{\ell-2}$ and if $u > 0$, any conformally related metric  $\widehat{h}= u^{p-2}h$ 
has volume form $\widehat{d\mu} = u^pd\mu$ and scalar curvature 
$$
\widehat{s} = u^{1-p} \left[(p+2) \Delta + s\right] u,
$$
where $\Delta = d^*d = -\nabla \cdot \nabla$, so  \eqref{james} implies that  any $\widehat{h}$ conformal to $h$  satisfies 
$$
\int_Z \widehat{s}\,  \widehat{d\mu}  = \int_Z [(p+2) |\nabla u|^2 + s_hu^2]\,d\mu_h\geq  \int_Z \left[2|\nabla u|^2 + 
s_h u^2\right]\, d\mu_h   > 0.
$$
It therefore suffices to recall that any conformal class $[h]$ on $Z$ contains a metric $\widehat{h} = 
u^{p-2}h$ whose scalar curvature does not change sign.
Here we could,  for example,  invoke Schoen's proof \cite{rick} of the Yamabe conjecture, and take 
 $\widehat{h}$ to be a Yamabe metric, which in particular has constant scalar curvature. 
 Or we could opt for the more elementary  trick, due to Trudinger \cite{trud}, of  just taking $u>0$ to belong to (and hence span)   the lowest eigenspace of the Yamabe Laplacian $(p+2) \Delta + s_h$.

 \bigskip 
 
 These ideas suffice  to provide a simple, clean proof of Lemma \ref{crux}. To see why, we begin by 
 introducing the following useful concept:

\begin{defn}
\label{expanse} 
A smooth  compact connected oriented $3$-manifold $N$  will be called {\sf expansive} if there is  
a smooth map $\phi : N\to V$ of non-zero degree to a   compact oriented 
aspherical $3$-manifold $V$ with $b_1(V)\neq 0$. 
\end{defn}

Here a compact manifold is called {\em aspherical} if   it is
an Eilenberg-MacLane space $K(\pi, 1)$. Thus,  a compact manifold is aspherical if and only if its universal cover
is contractible.  

\medskip

One   immediate consequence of Definition  \ref{expanse} is the following:
\begin{lemma}
\label{uno} 
If a smooth compact connected $3$-manifold $N^\prime$ admits a map of non-zero degree
to an expansive manifold $N$, then $N^\prime$  is itself expansive. 
\end{lemma}
\begin{proof} If $\psi : N^\prime \to N$ and $\phi : N\to V$, then $\deg (\phi \circ \psi) = \deg (\phi) \deg (\psi)$. 
\end{proof}

By contrast, the following key   consequence is distinctly   less trivial:
\begin{lemma}
\label{expansive}
Expansive $3$-manifolds never admit  metrics of positive scalar curvature. 
\end{lemma} 
\begin{proof} 
If $N$ is an expansive $3$-manifold, then, by definition, there is a smooth map $\phi : N\to V$ of $\deg (\phi ) \neq 0$,
where $V$ is a $K(\pi , 1)$ with $b_1(V) \neq 0$. Since $b_1(V)\neq 0$, there exists some non-zero  $\mathsf{a}\in H^1(V,\ZZ)$.
 Poincar\'e duality then guarantees the existence of some $\mathsf{b}\in H^2(V,\ZZ)$ with $\int_V \mathsf{a}\cup \mathsf{b}\neq 0$. 

We now proceed by  contradiction. Suppose  that $N$ admits a Riemannian metric $g$ 
of positive scalar curvature.  Let  $\mathscr{Z}$ be a mass-minimizing integral current representing 
the Poincar\'e dual of $\phi^*\mathbf{a}$ in $H_2(N,\ZZ)$, so that 
 $$\mathscr{Z}= \sum n_jZ_j$$
for some collection of compact embedded oriented surfaces $Z_j$ and integer coefficients $n_j$. 
Since $g$ has $s> 0$, the above second-variation argument implies that each $Z_j$ is a $2$-sphere. 
However, since $\pi_2 (V) =0$, this means that $\phi_* ([Z_j]) \in H_2 (V, \ZZ)$ must vanish for every $j$. 
Hence 
$$0 =\sum_j n_j \int_{\phi_*[Z_j]} \mathsf{b} =\int_{ \sum_j n_j  [Z_j]} \phi^*\mathsf{b}=\int_{N}  \phi^*\mathsf{a}\cup  \phi^*\mathsf{b}
= \deg (\phi ) \int_{V}  \mathsf{a}\cup  \mathsf{b} \neq 0.
$$
This  contradiction therefore shows  that such a metric cannot   exist.  
\end{proof}

Similar    ideas therefore now allow us to  deduce  the following result: 

\begin{thm} 
\label{taurus}
Let $N^3$ be an  {expansive}   $3$-manifold,
and let $X^4$ be a smooth compact
oriented $4$-manifold that admits a  smooth   submersion $\pip : X\to S^1$
 with fiber $N$. Let $P$ be any connected smooth  compact oriented $4$-manifold,
 and let $M= X\# P$. Then  $M$ does not admit Riemannian metrics of
 positive scalar curvature, and therefore   satisfies  $\mathscr{Y}(M)\leq 0$. 
\end{thm}

\begin{proof}
Let $\bd : X\# P\to X$ be the  smooth ``blowing down'' map that collapses $(P-B^4)\subset M$ to a point, and let 
$$f = \pip \circ \bd : M\to S^1= \mathbf{U}(1)$$
be  the induced projection.   We now  recall that 
\begin{equation}
\label{doh}
H^1(M, \ZZ) =  C^{\infty} (M, \mathbf{U}(1)) / \exp\,  [2\pi \mathsf{i}\, C^{\infty} (M, \RR)],
\end{equation}
because there is short exact sequence of sheaves of Abelian groups
$$0\to \ZZ\to C^\infty (\underline{\quad}, \RR)  \stackrel{\exp 2\pi \mathsf{i}\cdot }{\longrightarrow}   C^\infty (\underline{\quad}, \mathbf{U}(1)) \to 0,$$
where $C^\infty (\underline{\quad}, \RR)$ is a fine sheaf. 
The pre-image $f^{-1} (z )$ of a regular value is thus a copy
of $N$ in $M$ whose  homology class $[N] \in H_3(M, \ZZ)$ is 
Poincar\'e dual to $[f]\in  H^1 (M, \ZZ)$. Given any metric $g$ on $M$, we 
 now represent the homology class
$[N]$ by a mass-minimizing rectifiable current $\mathscr{Z}$, which can then be written as  a sum of 
smooth oriented connected compact hypersurfaces $Z_1, \ldots , Z_k$ with positive integer multiplicities; in particular, 
\begin{equation}
\label{play} 
[N] = \sum_{j=1}^k n_j[Z_j], \qquad n_j\in \ZZ^+.
\end{equation}
We now choose  a pairwise-disjoint collection of closed tubular neighborhoods $U_j\approx Z_j \times [-1,1]$ of the $Z_j$, and then  
express each  $Z_j$ as $f_j^{-1}(1)$  for a smooth map $f_j: M\to \mathbf{U}(1)$
that is $\equiv -1$  outside of  $U_j$, and given on $U_j$ by $f_j= e^{ i \Upsilon_j}$ for a 
smooth orientation-compatible defining function 
 $\Upsilon_j : U_j \to [-\pi, \pi]$  for $Z_j$ that is $\equiv \pm \pi$ 
 near each boundary component of $U_j$. Next, we define 
$$\widehat{f} : = \prod_jf_j^{n_j} : M\to  \mathbf{U}(1)$$
where the product is of course taken  point-wise in  $\mathbf{U}(1)$.
Since each $[f_j]$ is Poincar\'e dual to $[Z_j]$, it  follows from \eqref{play} that 
$[\widehat{f}]$ is therefore Poincar\'e dual to $\sum n_j [Z_j] =  [N]$. 
Since $f$ and $\widehat{f}$ therefore represent the same class in $H^1(M, \ZZ)$, equation 
\eqref{doh} therefore tells  that 
$$f = e^{2\pi \mathsf{i}w}  \widehat{f}$$
for some smooth real valued function $w: M\to \RR$, and we thus obtain an explicit homotopy of $\widehat{f}$ to $f$ by 
 setting $f_t = e^{2\pi \mathsf{i}tw}\widehat{f}$ for $t\in [0,1]$.
However,   $\widehat{f}$ is constant on each $Z_j$ by construction, so $f|_{Z_j}\sim \widehat{f}|_{Z_j}$ must induce the zero homomorphism
$\pi_1 (Z_j) \to \pi_1 (S^1)$, and  the inclusion map ${\incl}_j: Z_j\hookrightarrow M$ therefore 
lifts to an embedding  $\widetilde{\incl}_j:Z_j\hookrightarrow \widetilde{M}$ of this hypersurface in 
the covering space $\widetilde{M}\to M$ corresponding to  the kernel of $f_* : \pi_1(M) \to \pi_1(S^1)$. 


On the other hand, we can identify  $X$ with the mapping torus 
$$\taurus_\varphi := 
(N\times \RR) /\Big\langle \, (x, t) \longmapsto ( \varphi (x), t+2\pi ) \, \Big\rangle
$$
of some  diffeomorphism $\varphi: N\to N$ by simply 
 choosing a vector field on $X$ that projects to  $\partial/\partial \vartheta$ on $S^1$, and  then following the flow. 
Since  the diffeotype of  $\taurus_\varphi\to S^1$    only 
depends on the isotopy class of $\varphi$, we may also assume that $\varphi$
has a fixed-point $p\in N$. The flow-line $\{ p\} \times \RR$ then covers an embedded circle  $S^1 \hookrightarrow X$,
which we may moreover take to avoid the ball where  surgery is to be performed to  construct $M= X\# P$. 
This circle in $X$ then also defines  an embedded circle in $M$, which we will call the {\em reference circle}; 
and by first making a small perturbation, if necessary, we may assume that this reference circle $S^1\hookrightarrow M$
is also transverse to the $Z_j$. 
Since $N$ has intersection $+1$ with the reference circle, equation \eqref{play} tells us that at least one of the hypersurfaces 
$Z_j$ has non-zero intersection with the reference circle. 
Setting $Z= Z_j$ for some  such $j$,  
 we will  henceforth denote the corresponding inclusion map 
${\incl}_j$ by 
 ${\incl} : Z \hookrightarrow M$,  and its lift $\widetilde{\incl}_j$ by  $\widetilde{\incl} : Z \hookrightarrow \widetilde{M}$.
 Our mapping-torus model of $X$ now gives us a diffeomorphism $\widetilde{M}= (N\times \RR) \# ( \#_{\ell =1}^\infty P)$,
 along with  a blow-down map  $\widehat\bd  : \widetilde{M} \to N\times \RR$ that lifts 
 $\bd  : {M} \to X$. However,    $\{ p \} \times \RR\subset N \times \RR$ now meets $\widehat\bd \circ \widetilde{\incl} (Z )$ transversely in a set
 whose oriented count 
 exactly computes the homological intersection number $\mathfrak{n}\neq 0$
of our reference circle $S^1$ with $Z$. 
Thus,  if $\proj : N\times \RR \to N$ denotes the first-factor projection, 
the smooth map $\proj\circ  \widehat\bd \circ \widetilde{\incl} : Z \to N$ has  degree $\mathfrak{n}\neq 0$.

Since $N$ is expansive by hypothesis, this shows that $Z$ is also expansive by Lemma \ref{uno}. 
Hence the given metric  $g$ cannot possibly have positive scalar curvature. Indeed, if it did,
the induced metric $h$ on $Z=Z_j$ would be conformal to a metric $\widehat{h}$ of positive scalar curvature by the second-variation 
argument detailed above. But  since $Z$ is expansive, Lemma \ref{expansive} forbids the existence of  a positive-scalar-curvature metric
$\widehat{h}$ on $Z$. Hence $M$ cannot admit a metric $g$ of positive scalar curvature either,  and 
$\mathscr{Y}(M) \leq 0$, as claimed. 
\end{proof}

Theorem \ref{taurus} now immediately  implies Lemma \ref{crux}. 
Indeed, let $N\to \Sigma$ be any   non-trivial circle bundle over
a  hyperbolic Riemann surface. Then $N^3$  is  aspherical, because its
 universal cover $\mathcal{H}^2 \times \RR\approx \RR^3$ is contractible. 
 Moreover,  $b_1(N) = b_1(\Sigma) \neq 0$. Hence $N$ is expansive, because the 
  identity  $N\to N$ is  a degree-one map to a $K(\pi , 1)$ with $b_1\neq 0$. 
   We next   set  $X = N\times S^1$, so that the second-factor projection $\pip : X \to S^1$
 is then a submersion with fiber $N$.  Finally, letting $P$ be the  connected sum of $k$ copies of $\overline{\CP}_2$, 
we observe  that  the manifold $M= (N\times S^1)\# k \overline{\CP}_2$  of  Lemma \ref{crux}  is
one of the $4$-manifolds covered by  Theorem \ref{taurus}.  Since Lemma \ref{crux} implies 
Theorems \ref{ellipsis} and \ref{parity}  by 
Proposition \ref{lynx},  this provides one  complete    proof of our main results.

  Nonetheless, by invoking  Perelman's proof \cite{bbbpm,lott} of Thurston's geometrization conjecture, we can    improve  Theorem \ref{taurus} to yield the following {\sf sharp} result:

  \begin{thm} 
\label{sharp}
Let $N^3$ be compact oriented  connected   $3$-manifold,
and let $X^4$ be a smooth compact
oriented $4$-manifold that admits a  smooth   submersion $\pip : X\to S^1$
 with fiber $N$. Let $P$ be any connected smooth  compact oriented $4$-manifold,
 and let $M= X\# P$. Then $\mathscr{Y}(N) \leq 0\, \Longrightarrow \, \mathscr{Y}(M)\leq 0$. 
\end{thm}

To prove this, we  just need  to first demonstrate the following:  

\begin{prop}
\label{better} 
If $\phi : Z\to N$ is a map 
of non-zero degree between  compact oriented  connected   $3$-manifolds, then  $\mathscr{Y}(Z) >  0
\,\Longrightarrow\,  \mathscr{Y}(N) > 0$.
\end{prop}

\begin{proof}
Perelman's  Ricci-flow proof of Thurston's 
uniformization conjecture implies \cite{bbbpm,lott} that a closed oriented $3$-manifold $Z$ admits
a metric of positive scalar curvature iff it is a connected sum of spherical spaces forms $S^3/\Gamma_j$
and/or copies of $S^2 \times S^1$. Geometrization   then goes on to tell  us that 
  a closed oriented $3$-manifold $N$ that does not admit metrics
of positive scalar curvature must therefore be expressible as a connected sum $N= N_0\# V$, where $V$ is a $K(\pi , 1)$
and $N_0$ is some other  $3$-manifold. Consequently, if $\mathscr{Y}(N) \leq 0$, there is  a degree-one map 
$\psi: N\to V$ to some compact oriented aspherical $3$-manifold  $V$ obtained by collapsing $N_0-B^3$ to a point. 

Arguing by contradiction, we now assume  that there exists some  map 
$\phi : Z\to N$ of non-zero degree, where $\mathscr{Y}(Z) > 0$ but  $\mathscr{Y}(N) \leq 0$.
Letting $\psi: N\to V$ be a degree-one map to an aspherical manifold, we thus obtain 
 a map $\Psi:= \psi \circ \phi  : Z \to V$ of non-zero degree, where $V$ is a  $K(\pi , 1)$ and 
$$
Z \approx (S^3/\Gamma_1)  \# \cdots \# (S^3/\Gamma_k) \# \underbrace{(S^2\times S^1) \# \cdots \#(S^2\times S^1)}_\ell 
$$
for some $\ell \geq 0$. The latter may be restated as saying that  $\coprod_{j=1}^k (S^3 /\Gamma_j)$ can be  obtained from $Z$ by performing surgeries  in dimension $2$. Since every surgery can be 
realized by a cobordism, this means that there is a compact oriented $4$-manifold-with-boundary $W$ with 
$\partial W = \overline{Z}\sqcup \coprod_{j=1}^k (S^3/\Gamma_j)$, and such that  $W$ deform-retracts to a space
obtained from ${Z}$ by attaching $3$-disks $D^3$ along their boundaries at  $\max (\ell, k+\ell-1)$ disjoint embedded $2$-spheres in  $Z$. 
Since $\pi_2(V) =0$, we can therefore extend $\Psi$ across these $3$-disks, and thereby obtain a 
 map $\widehat{\Psi} :  W\to V$ with  $\widehat{\Psi}|_{\overline{Z}} = \Psi$. If we  define $\Psi_j: S^3 /\Gamma_j \to V$  to be  
the restriction of 
 $\widehat{\Psi}$ to the corresponding  component of $\partial W$, 
we thus have 
$$-\Psi_* [Z] + \sum_{j=1}^k \Psi_{j*} [S^3 /\Gamma_j]= \widehat{\Psi}_* [\partial W]= 0 \in H_3 (V, \ZZ),$$
and pairing with the generator of $H^3(V,\ZZ)$ therefore yields 
 \begin{equation}
 \label{hummer}
\deg (\Psi)  = \sum_{j=1}^k \deg (\Psi_j ).
\end{equation}
 However, because $\pi_3 (V)=0$, any 
 map $S^3 \to V$ is homotopic to a constant, and therefore  has degree $0$; and 
since mapping-degree is multiplicative under compositions, this in turn  implies  that every  map $S^3/\Gamma_j \to V$ also has 
  degree zero. Thus   $\deg (\Psi_j)=0$ for all $j$, and  hence $\deg (\Psi )=0$ by 
 \eqref{hummer}. Since this contradicts the fact that    $\deg (\Psi )\neq 0$ by construction, 
our assumption was therefore false. Hence  $\mathscr{Y}(Z) >  0
\,\Longrightarrow\,  \mathscr{Y}(N) > 0$ whenever there is a map $\phi: Z\to N$ of non-zero degree. 
   \end{proof} 

The proof of Theorem \ref{sharp} is now identical to the proof of Theorem \ref{taurus},
except that,  in the very last paragraph, 
 Proposition \ref{better} now supplants Lemmas \ref{uno} and \ref{expansive}.
 Our proof of  Lemma \ref{crux} can then be rephrased in terms of 
  Theorem \ref{sharp} 
 by simply observing   that  our proof of Proposition  \ref{better} also proves the following: 
  
  \begin{prop}
  \label{clearer} 
  Let $N$ be a compact oriented $3$-manifold that admits a map $\phi : N \to V$ of non-zero degree to an aspherical manifold. 
  Then $\mathscr{Y} (N) \leq 0$. 
  \end{prop}

   This   sharpens Lemma \ref{expansive} by dropping the requirement  that $b_1(V)\neq 0$. However, since 
   recent  results of Agol, Wise, and others \cite{agol-icm} show  that  compact  aspherical  $3$-manifolds  always 
       have finite-sheeted 
   covers with $b_1\neq 0$,  the difference here is arguably more a matter of {\ae}sthetics than of substance. 
In any case,    the pioneering papers of  Schoen-Yau \cite{sy3man} and  Gromov-Lawson  \cite{gvln2}, 
which proved  strikingly similar   results  by   entirely 
   different methods, provide attractive alternatives for  addressing the $3$-dimensional aspects of our story.    
   
   Because the results of this section have been  formulated in much greater generality than what we'd     need to  just
 prove Lemma \ref{crux}, they of course  have many other applications.  For example, we will later see, in \S \ref{pathology} 
 below, that they  also imply  interesting results concerning  complex surfaces with 
  $b_1$ odd and  $\kod =-\infty$. For related results, 
  including a non-Seiberg-Witten   proof   of the  $b_1$ odd, $\kod =0$ case of Theorems 
  \ref{ellipsis} and \ref{parity}, see   \cite[Theorem 4.8]{alba}.

\pagebreak

\section{Proof via the Seiberg-Witten Equations} 
\label{monopoly} 

In this section, we give an entirely  different  proof of Lemma \ref{crux} based on 
an indirect use of the Seiberg-Witten equations pioneered  by Kronheimer \cite{K}.

Let $M$ be a $4$-manifold with a fixed spin$^c$ structure $\mathfrak{c}$. For any metric $g$ on $M$, 
this gives us two  rank-$2$ complex vector bundles $\mathbb{V}_\pm \to M$,  along with formal isomorphisms
$$\mathbb{V}_\pm= \mathbb{S}_\pm\otimes L^{1/2}$$
where $\mathbb{S}_\pm$ are the (locally-defined) 
spin bundles of $(M,g)$, and where  $$L =  \det (\mathbb{V}_+) =   \det (\mathbb{V}_-) $$
is called the  determinant line bundle of the of the spin$^c$ structure. Since   
$$c_1(L) \equiv w_2 (M) \bmod 2,$$
 the square-root $L^{1/2}$, like the vector bundles $\mathbb{S}_\pm$, is not   globally defined if $w_2 (M)\neq 0$; however, 
the obstruction to  consistently choosing  the signs of the   transition functions for $L^{1/2}$  exactly cancels the 
obstruction to consistently choosing   signs for the 
transition functions of  $\mathbb{S}_\pm$. 

The {\em Seiberg-Witten equations} are  equations for a pair $(\Phi, \theta)$, where $\Phi$ is a section of $\mathbb{V}_+$ and 
$\theta$ is a $\mathbf{U}(1)$-connection on $L$. They  are given by 
 \begin{eqnarray} \dir_\theta\Phi &=&0\label{drc}\\
 F_\theta^+&=&-\frac{1}{2} \Phi \odot \overline{\Phi} \label{sd}\end{eqnarray}
 where $\dir_\theta : \Gamma (\mathbb{V}_+)\to \Gamma (\mathbb{V}_-)$ is the Dirac operator
 coupled to $\theta$, and $F_\theta^+$ is the self-dual part of the curvature of $\theta$, 
while  $\odot$ denotes the symmetric tensor product,  and  we have
 used the identification $\Lambda^+\otimes \CC = \odot^2 \mathbb{S}_+$ induced by Clifford multiplication. 
 The Seiberg-Witten equations  (\ref{drc}--\ref{sd})  then imply   the  Weitzenb\"ock formula 
  \begin{equation}
 0=	2\Delta |\Phi|^2 + 4|\nabla_{\theta}\Phi|^2 +s|\Phi|^2 + |\Phi|^4 ,	
 	\label{wnbk}
 \end{equation}
which  leads to striking results relating    curvature  and  differential topology.

On  complex surfaces of  K\"ahler type, Seiberg-Witten theory plays a  crucial role in the 
calculation of  Yamabe invariants. If  a $4$-manifold $M$  with  $b_+ > 1$  admits an 
almost-complex structure,   
Witten \cite{witten} defined  an invariant that counts the expected  number of gauge-equivalence classes of solutions of the equations
 with multiplicities, 
relative to the spin$^c$ structure arising from an almost-complex structure. When this count is non-zero,
the first Chern  class $c_1(L)$ of the spin$^c$ structure is
 then called a {\em basic class}. One of Witten's most revolutionary   discoveries was that 
if $(M,J)$ is a compact K\"ahler  surface with $b_+ > 1$, then $c_1(M,J)$ is a basic class. 
This in particular implies that no such $M$  can admit metrics of positive scalar curvature. 

However, Seiberg-Witten basic classes do not provide a workable method for  proving Lemma \ref{crux}.
Indeed,  a theorem of  Biquard \cite[Th\'eor\`eme 8.1]{biq2} shows, for example,  
 that  $S^1 \times N$ 
carries no basic classes
if  $N\to \Sigma$ is a non-trivial circle bundle over a hyperbolic Riemann surface. 
It might therefore be tempting to think that the Seiberg-Witten 
equations have no bearing  at all on properly elliptic surfaces  of non-K\"ahler type.

However, there are \cite{baufu,ozsz} still  many $4$-manifolds without basic classes where, for   specific spin$^c$ structures,  
the Seiberg-Witten 
equations  nonetheless admit solutions for every Riemannian metric. 
Such a  state of affairs is usually described  by using  the following 
convenient definition, which was  first introduced by Peter Kronheimer \cite[Definition 5]{K}:
  
\begin{defn}
Let $M$ be a smooth compact oriented $4$-manifold
with $b_{+}\geq 2$. An element $\mathbf{a}\in  H^{2}(M,\ZZ )/
\mbox{\rm torsion}$,  is  called a {\sf monopole
class} of $M$ iff there is some  spin$^{c}$ structure
$\mathfrak{c}$ 
on $M$ with first Chern class 
$$c_{1}(L)\equiv \mathbf{a} ~~~\bmod \mbox{\rm torsion}$$ for which   the   Seiberg-Witten 
equations (\ref{drc}--\ref{sd})
have a solution for every Riemannian  metric $g$ on $M$. 
\end{defn}

Unfortunately, however, the   expositional prominence accorded to this idea  appears to have resulted in  
 widespread  misunderstanding   of  Kronheimer's paper \cite{K},
which   ultimately   never   claimed  to prove the existence of 
monopole classes on the $4$-manifolds it studied.  We will  therefore  need to  carefully
discuss  Kronheimer's argument    in order to    highlight   what it 
actually proves.

Let us now    recall that an integral cohomology class $\mathbf{a}\in  H^{2}(M,\ZZ )/ \mbox{\rm torsion}$
is said to be {\em characteristic} if 
$$\mathbf{a}\bullet \mathbf{b} \equiv      \mathbf{b}\bullet \mathbf{b}\bmod 2 ,
\qquad \forall\mathbf{b}\in H^{2}(M,\ZZ )/\mbox{\rm torsion},$$ 
where $\bullet$ denotes  the  intersection pairing.  For any spin$^c$ structure, 
the image  $c_1^\RR (L)\in H^{2}(M,\ZZ )/ \mbox{\rm torsion}\subset H^2(M, \RR)$ of $c_1(L)$ is a characteristic element, and 
conversely,   every characteristic element  arises in this way from some spin$^c$ structure on $M$. 

 When $\mathbf{a}$ is a monopole class, applying the Weitzenb\"ock formula \eqref{wnbk} to  solutions of the Seiberg-Witten equations
 implies a scalar-curvature integral  estimate involving $\mathbf{a}$ that must hold for every metric on the manifold. 
 Without even mentioning the Seiberg-Witten equations, however, it is  possible
  to simply axiomatize the property   that 
 this estimate holds  for every possible metric. Doing so then leads to  the following definition:

\begin{defn}
\label{effective}
Let $M$ be a smooth compact oriented $4$-manifold
with $b_{+}\geq 2$.  A characteristic integral  cohomology class  
$\mathbf{a}\in  H^{2}(M,\ZZ )/ \mbox{\rm torsion}\subset H^2 (M,\RR)$  will be called 
a {\sf mock-monopole class}  of $M$ if  every Riemannian metric $g$ on $M$
satisfies  the inequality 
\begin{equation}
\label{virtue}
\int_M s_-^2 ~ d\mu_g \geq 32\pi^2 [\mathbf{a}^+]^2,
\end{equation}
where $s_-(x):= \min (s_g(x), 0)~\forall x\in M$, and where  $\mathbf{a}^+\in H^2(M, \RR)$  is 
the orthogonal projection  of  $\mathbf{a}$, with respect to the intersection form $\bullet$,  to the 
$b_+(M)$-dimensional subspace $\mathcal{H}^+_g\subset H^2 (M,\RR)$
consisting of  those deRham classes that  are represented by 
  self-dual  harmonic
$2$-forms with respect to $g$. 
\end{defn}

We will see in a moment  that every monopole class is also a mock-monopole class. However, 
 Kronheimer's paper \cite{K} develops a  less direct way of using  the Seiberg-Witten equations 
 to show that certain $4$-manifolds carry mock-monopole classes. Before discussing this further, however, 
 let us first observe that Definition \ref{effective} has some pertinent 
 immediate consequences.

\begin{prop}
\label{thumper} 
Let $M$ be a smooth compact oriented $4$-manifold
with $b_{+}\geq 2$. If $M$ carries a {\sf non-zero} mock-monopole class, then $\mathscr{Y}(M)\leq 0$. 
\end{prop}
\begin{proof}
Since the harmonicity of  $2$-forms on $M^4$ is unaffected   by conformal changes of metric $g\rightsquigarrow \widehat{g}= u^2g$,
and since the self-dual/anti-self-dual decomposition $\Lambda^2 = \Lambda^+_g \oplus \Lambda^-_g$ of $2$-forms 
 is also conformally invariant, the 
self-dual subspace $\mathcal{H}^+_g=  \mathcal{H}^+_{[g]}$ of $H^2(M, \RR)$ only depends on the conformal class $[g]$ of the 
metric $g$. Since the intersection form $\bullet$ is positive-definite on  $\mathcal{H}^+_{[g]}$ for any 
conformal class $[g]$, we therefore have $[\mathbf{a}^+]^2 > 0$ whenever $\mathbf{a}^+= (\mathbf{a}^+)_g\neq 0$. Thus, 
if a conformal class $[g]$ satisfies $(\mathbf{a}^+)_g\neq 0$, its Yamabe constant
$$Y(M, [g]) = \inf_{\widehat{g}\in [g]} \frac{\int_M \widehat{s}~\widehat{d\mu}}{\sqrt{\int_M \widehat{d\mu}}}$$
must  satisfy $Y(M, [g]) <0$, because  otherwise $[g]$ would contain a metric $\widehat{g}$ with $\widehat{s}\geq 0$, and hence  with 
$\widehat{s}_-\equiv 0$, 
and this would then  violate the key    inequality \eqref{virtue} that must be satisfied by any mock-monopole class. 

On the other hand, because  we  have $\mathbf{a}\neq 0$ by hypothesis, the set of metrics $\{ g ~|~ (\mathbf{a}^+)_g\neq 0\}$ 
is {\em dense} in the $C^2$ topology, as a consequence of  the fact \cite[Proposition 4.3.14]{don} that  the period map 
\begin{eqnarray*}
\{ C^2 \mbox{ metrics}\} &\longrightarrow& Gr_{b_+}^+ (H^2(M,\RR))\\
g \qquad&\longmapsto & \qquad \mathcal{H}^+_g
\end{eqnarray*}
 is everywhere  transverse   to the set of positive planes orthogonal to $\mathbf{a}$.  
 But  the Yamabe constant $Y(M, [g])$ is a continuous   function of 
$g$ in the $C^2$ topology  \cite[Proposition 4.31]{bes}, so
 taking limits then shows  that $Y(M, [g])\leq 0$ for every metric $g$. Consequently,
 $\mathscr{Y} (M) = \sup_g Y(M, [g])\leq 0$, as claimed. 
\end{proof}

\begin{cor} 
\label{bumper} 
Let $X$ be a smooth compact oriented $4$-manifold with $b_+\geq 2$, and let 
$M = X\# k \overline{\CP}_2$ for some $k \geq 1$. If $M$ admits a mock-monopole class,
then neither $M$ nor $X$ can admit metrics of positive scalar curvature. 
\end{cor} 
\begin{proof} Since the  Mayer-Vietoris sequence gives us a canonical isomorphism 
 $H^2(X\# k \overline{\CP}_2) = H^2 (X) \oplus [H^2(\overline{\CP}_2)]^{\oplus k}$, 
 we may take $\mathsf{E}\in  H^2(M, \ZZ)/\mbox{torsion}$  to be 
a  generator 
of one $H^2 (\overline{\CP}_2, \ZZ)$ summand  in $H^2(M,\ZZ)/\mbox{torsion}$. 
We then have $\mathsf{E} \bullet \mathsf{E}=-1$. 
If $\mathbf{a}\in H^2(M, \ZZ)/\mbox{torsion}$ is a mock-monopole class, the fact that 
$\mathbf{a}$ is characteristic then implies that 
$$\mathbf{a} \bullet \mathsf{E} \equiv \mathsf{E} \bullet \mathsf{E} \equiv  1  \bmod 2.$$
It therefore follows that $\mathbf{a}\neq 0$, and  Proposition \ref{thumper} therefore tells us
that $M$ cannot admit metrics of positive scalar curvature.

However,  the Gromov-Lawson surgery lemma \cite{gvln} implies that the connected sum 
of two positive-scalar-curvature $4$-manifolds also admits metrics of positive scalar curvature.  Since $\overline{\CP}_2$ admits
metrics of positive scalar curvature, while $M = X\# k \overline{\CP}_2$
does not, it therefore follows that   $X$ {cannot} admit metrics of positive scalar curvature either. 
\end{proof}

Before we discuss Kronheimer's argument, let us first see why any monopole class is 
also a mock-monopole class. If  $M$ carries a monopole class $\mathbf{a}$, 
then there is a spin$^c$ structure $\mathfrak{c}$ with  $c_1^\RR (L)=\mathbf{a}$
and relative to which there is some  solution $(\Phi, \theta )$ of the Seiberg-Witten equations (\ref{drc}--\ref{sd})
for  each metric $g$. The Weitzenb\"ock formula \eqref{wnbk} therefore tells us that 
\begin{eqnarray*}
 0&=	&2\Delta |\Phi|^2 + 4|\nabla_{\theta}\Phi|^2 +s|\Phi|^2 + |\Phi|^4 \\
 &\geq& 	2\Delta |\Phi|^2 +(s_-) |\Phi|^2 + |\Phi|^4 
\end{eqnarray*}
and integration therefore tells us that 
$$\int_M (-s_-) |\Phi|^2~ d\mu_g \geq \int_M  |\Phi|^4~ d\mu_g .$$
Applying the Cauchy-Schwarz inequality, we thus have
$$\left(\int_M s_-^2~ d\mu_g
\right)^{1/2} \left(\int_M |\Phi|^4~ d\mu_g
\right)^{1/2} \geq \int_M |\Phi|^4~ d\mu_g$$
and squaring therefore gives us 
$$\int_M  s_-^2~ d\mu_g \geq \int_M |\Phi|^4~ d\mu_g= 8  \int_M |F_\theta^+|^2~ d\mu_g$$
where the last equality is an algebraic consequence of  \eqref{sd}. Because $iF^+_\theta$ differs
from the harmonic representative of $2\pi  (\mathbf{a}^+)_g$ by the self-dual part of an exact form,
we also have $$\int_M |F_\theta^+|^2~ d\mu_g\geq 4\pi^2 [(\mathbf{a}^+)_g]^2,$$
and putting these two inequalities together then shows that \eqref{virtue} is therefore satisfied for
the given metric $g$.  Since this argument applies equally well to   every other metric, it therefore follows that 
$\mathbf{a}$ satisfies the definition of  a mock-monopole class.  

While the above argument is nearly standard in Seiberg-Witten theory, 
typical uses of it \cite{K,lno}  never even mention the function $s_- = \min (s_g, 0)$,
and instead just  use it to   deduce    lower bounds
for the important Riemannian functional 
$$\int s_g^2 \, d\mu_g \geq  \int s_-^2 \, d\mu_g .$$
However, given the crucial role played by $s_-$ in the proof of Proposition \ref{thumper}, 
our objectives require us to emphasize this subtle refinement of the conclusion. 

With these ideas established, we now come to Kronheimer's construction. Let $N$ be  a compact connected oriented 
prime $3$-manifold 
with $b_1\geq 2$ that  is equipped with a {\em taut foliation}. Prime means that $N$ cannot be expressed as
a non-trivial connected sum, and the assumption that $b_1(N) \geq 2$ implies, in particular, that $N$ is not $S^2\times S^1$. A taut foliation amounts  to   a Frobenius-integrable distribution 
 $D\subset TN$ of oriented $2$-planes for which there exists a smooth  closed curve 
$S^1\subset Y$ that  meets every integral surface of $D$ transversely. 
We now set $X= N\times S^1$, and  define an orientation-compatible almost-complex structure $J^\prime$ on $X$ 
 that, with respect to some arbitrary Riemannian metric $h$ on $N$, sends $D^\perp\subset TN$ to $TS^1$, 
while acting on $D$ by $+90^\circ$ rotation. Deforming $J^\prime$ to be integrable near certain points and then 
blowing up then equips $M:= X \# k \overline{\CP}_2$ 
with an almost-complex structure $J$ in a specified  manner. While there are
 many choices involved in this construction, the homotopy class of the resulting  $J$  is independent
of these choices, so this construction  equips $M$ with a preferred spin$^c$ structure.  With one modest improvement, 
and using  the terminology we have just introduced, 
the proof of  \cite[Proposition 8]{K} then actually proves the following:

\begin{prop}[Kronheimer] 
\label{kron} 
Let $N$ be a compact oriented  connected prime  $3$-manifold  with $b_1(N) \geq 2$
that carries an  integrable oriented distribution  $D\subset TN$  of  $2$-planes that is  tangent to a taut foliation. 
Set   $X= N\times S^1$,   and  equip
$M= X\# k \overline{\CP}_2$ with the almost-complex structure $J$ described  above. 
Then 
$\mathbf{a} := c_1^\RR (M, J) \in H^2(M, \ZZ)/\mbox{torsion}$ is a mock-monopole class. 
\end{prop}

Let $X_\ell= N\times[0,\ell]$,  set $M_ \ell := X_\ell \# (k\ell) \overline{\CP}_2$, and let $\overleftrightarrow{M}_\ell$ be obtained
from $M_\ell$ by identifying $N \times\{ 0\}$ with $N \times\{ \ell\}$. Kronheimer begins by observing that 
$\overleftrightarrow{M}_\ell$ can be thought of as an $\ell$-fold cover of $M$. Given a metric $g$ on $M$,
we can  therefore pull it back to $\overleftrightarrow{M}_\ell$ and $M_\ell$ as a metric ${g}_\ell$, and then consider 
the Seiberg-Witten equations (\ref{drc}--\ref{sd}) 
with respect to 
${g}_\ell$  and the pulled-back spin$^c$ structure. Because the tautness hypothesis implies
 that $D$ can deformed into a
contact structure of either sign, Kronheimer shows that $M_\ell$ admits a symplectic form 
that is convex at both ends, allowing him to view it as the central region of 
 an
 asymptotically conical almost-K\"ahler manifold obtained by adding conical ends that do
not depend on $\ell$. This allows him to predict the existence of  solutions of the Seiberg-Witten equations on ${M}_\ell$ 
that are uniformly controlled in the regions near the boundary. By cutting  these off and  pasting,
he  then obtains pairs $(\Phi , \theta )$ on $(\overleftrightarrow{M}_\ell,g_\ell)$  that satisfy    some perturbation  
\begin{eqnarray*}
\dir_\theta \Phi &=& \Omega\\
F_\theta^+ &=&-\frac{1}{2}\Phi \odot \bar{\Phi} + i \eta
\end{eqnarray*}
of the Seiberg-Witten equations, 
where $\Omega\in \Gamma (\mathbb{V}_-)$ and   $\eta\in \Gamma (\Lambda^+)$ are supported
 in $(N\times [0, \epsilon]) \cup (N\times [\ell -  \epsilon, \ell])$ and  satisfy   uniform point-wise  bounds that are independent of $\ell$. 
 By integrating the Weitzenb\"ock formula 
$$ 4|\! \dir_\theta\Phi |^2 =	4|\nabla_{\theta}\Phi|^2 +s|\Phi|^2 + |\Phi|^4 + 4 \langle \eta , i\Phi \odot \bar{\Phi} \rangle +  \mbox{divergence terms}$$
on $\overleftrightarrow{M}_\ell$, this yields 
$$\int_{\overleftrightarrow{M}_\ell} \left[4 |\Omega|^2 + ( 2\sqrt{2}|\eta| -s_-) |\Phi |^2\right]  d\mu_{g_\ell}\geq \int_{\overleftrightarrow{M}_\ell}  |\Phi|^4 d\mu_{g_\ell}$$
and the Cauchy-Schwarz and  triangle inequalities therefore yield 
$$C+ \left( \int_{\overleftrightarrow{M}_\ell} s_-^2 \, d\mu_{g_\ell} \right)^{1/2}  \geq  \left(8 \int_{\overleftrightarrow{M}_\ell} |\mathcal{F}^+|^2 \, d\mu_{g_\ell} \right)^{1/2}$$
where $C$ is a constant
independent of $\ell$, and where $\mathcal{F}^+$ is the pull-back to $\overleftrightarrow{M}_\ell$ of the self-dual part of the 
harmonic representative of $2\pi c_1^\RR (M, J)= 2\pi \mathbf{a}$ 
with respect to $g$. Back down on $(M,g)$, however, this is equivalent to the statement that 
$$ \frac{C}{\sqrt{\ell}}+  \left( \int_{M} s_-^2 \, d\mu_{g} \right)^{1/2}  \geq  \left( 32\pi^2 [\mathbf{a}^+]^2 \right)^{1/2}$$
and since $C$ is independent of $\ell$, we therefore obtain the desired inequality \eqref{virtue} by taking the limit as  $\ell \to \infty$. 
Since this works for any metric $g$ on $M$, it follows that $\mathbf{a}=c_1^\RR(M,J)$ is therefore  a mock-monopole class, as claimed. 

\bigskip 

This provides 
a Seiberg-Witten proof of  Lemma \ref{crux}. Indeed,  notice that, by passing to covers, it suffices to consider 
the case when $N\to \Sigma$ is  a circle bundle of Euler class $+1$ over  a hyperbolic surface. Since
a construction due to  Milnor \cite{milnor-flat} endows $N\to \Sigma$ with a flat $\mathbf{SL}(2, \RR)$-connection, 
$N$ therefore admits a taut foliation; and since   
  $b_1 (N) = b_1(\Sigma ) > 2$, Proposition \ref{kron} therefore applies. 
Thus 
 $M = (N\times S^1) \# k \overline{\CP}_2$ carries a mock-monopole class for any $k\geq 0$, 
and  Corollary \ref{bumper}  therefore asserts that neither  $X=N\times S^1$ nor $M = X \# k \overline{\CP}_2$ can admit metrics of positive-scalar curvature, as claimed. Theorems \ref{ellipsis} and \ref{parity} then once again  follow by Proposition \ref{lynx}. 

\pagebreak 

\section{Pathological Features  of  Class VII} 
\label{pathology} 

%

Our formulation of  Theorems \ref{ellipsis} and \ref{parity} has intentionally excluded  surfaces
with $\kod = -\infty$ and  $b_1$ odd. These complex surfaces, which actually  all have $b_1=1$, are 
for historical reasons known   as 
{\sf surfaces of class  {VII}}. The most familiar  class-{\sf VII} surfaces  
are the (primary) Hopf surfaces $(\CC^2 - \{ 0\})/\ZZ$, which are diffeomorphic to 
$S^3\times S^1$. These are already  pathological from the standpoint of Theorem
\ref{parity}, because results of Kobayashi \cite{okob} or Schoen  \cite{sch}  imply that a primary Hopf surface has Yamabe invariant 
$\mathscr{Y} (S^3\times S^1)= \mathscr{Y} (S^4)=8\sqrt{6}\pi$, 
while theorem of Gursky-LeBrun \cite{gl1} shows that  its one-point blow-up has $\mathscr{Y}([S^1 \times S^3] \# \overline{\CP}_2) = \mathscr{Y} (\CP_2) = 12\sqrt{2}\pi$.  
 The exclusion of class {\sf VII} surfaces from Theorem \ref{parity} is therefore a matter of necessity. 
 
 However,  surfaces of class {\sf VII} must also be excluded from Theorem  \ref{ellipsis},
 because their Yamabe invariants are not always of the same sign. Indeed, while the Hopf surfaces discussed
 above  have positive Yamabe invariant,  a  class of  minimal class-{\sf VII} surfaces 
 discovered by Inoue \cite{inoue72}
  were  shown by  the first author  \cite{alba} to have Yamabe invariant zero. 
 We will call these examples  Inoue-Bombieri
surfaces,  both  because Inoue's paper credited  Bombieri with their independent discovery, 
and  to distinguish them from various other class-{\sf VII} surfaces that also bear Inoue's name. 
 As yet another    application of Theorem \ref{taurus},  we now give a simplified proof of \cite[Theorem 4.5]{alba}: 
 
\begin{thm}[Albanese]
\label{michael} 
Let $X$ be an Inoue-Bombieri  surface,  and let $M$ be obtained from 
$X$ by blowing up $k\geq 0$ points. Then $\mathscr{Y}(M)=0$. 
\end{thm}
\begin{proof} Each Inoue-Bombieri surface $X$ is the  quotient of a half-space in $\CC^2$ by a discrete group of
affine transformations that preserves the half-space. 
Each one comes  equipped  \cite{inoue72} with a smooth submersion $\pip : X\to S^1$ whose fibers are 
aspherical  $3$-manifolds $N$ with $b_1\neq 0$;  more specifically,  
 $N$ is always either  the  $3$-torus $\mathbb{T}^3$ or a 
circle bundle 
$N\to \mathbb{T}^2$ over the $2$-torus. 
Since  each such  $N$  is  expansive, Theorem  \ref{taurus} tells us   that $\mathscr{Y}(M)\leq 0$. 
On the other hand, each such $X$ admits an $F$-structure of positive rank, in the sense of 
Cheeger-Gromov \cite{cheegro}, so each such $X$ has $\mathscr{Y}(X)\geq 0$, 
and Kobayashi's connect-sum theorem \cite[Theorem 2]{okob} therefore implies that $\mathscr{Y}(M)= \mathscr{Y}(X\# k \overline{\CP}_2)\geq 0$, too. 
We therefore conclude that  $\mathscr{Y}(M) =0$, as claimed. 
\end{proof}

Because class-{\sf VII} surfaces inhabit a world  where algebraic geometry holds so little sway, it is perhaps unsurprising that 
they turn out to violate so many of our K\"ahlerian expectations. The question of what new patterns we might  discover here 
 is a matter deserving  further discussion  in  the next section. 

 \pagebreak

\section{Problems and Perspectives} 
\label{conclusion}
 
 A main theme of this article has been that Lemma \ref{crux} can actually be
 proved in many different ways, and that once this lemma is known, our  main results,  
 Theorems \ref{ellipsis} and \ref{parity}, then follow. 
 In \S \ref{schoen-yau},   we gave  a self-contained 
 explanation of  how the  Schoen-Yau method works, and  then showed why it implies our key results.
  However, a recent paper 
   of Gromov \cite{misha} gives a systematic approach to the subject that subsumes  all of the essential arguments used in 
   \S \ref{schoen-yau}.     Furthermore, an unpublished argument by Jian Wang implies   that
   positive-scalar-curvature 
 metrics cannot exist on $4$-manifolds that admit maps   of non-zero degree to  aspherical $4$-manifolds with $b_1\neq 0$,
  and  
 a recent preprint of  Chodosh, Li, and Liokumovich \cite{chadosh} improves this by dropping the 
 $b_1 \neq  0$ hypothesis. 
Either result suffices to prove  Lemma  \ref{crux} and Theorem \ref{michael}.

While the Seiberg-Witten proof of Lemma \ref{crux} given in \S \ref{monopoly}  emphasizes the degree to which 
 mock-monopole classes  are   quite  sufficient 
for  our purposes,  it is entirely possible that Proposition \ref{kron}  might actually reflect the existence of a true monopole class 
on some high-degree cover of the manifold. Any result in this direction would certainly shed completely new  light on  the  subject.

 Many intriguing  questions  about the Yamabe invariant remain open for complex  surfaces of  $\kod = -\infty$.
 For example, while the Yamabe invariant is always  positive when such a complex surface has  $b_1$  even, 
 we still do not know whether blowing up  a rational or ruled surface  can ever  change the precise value of its Yamabe invariant; 
 all we know for sure is that the invariant must remain confined to a relatively narrow numerical range. For complex surfaces
of class {\sf VII}, the situation is even more daunting; because a full classification of these manifolds is lacking,
we still  cannot be absolutely certain that a  class-{\sf VII} surface can never have negative Yamabe invariant! 
Still, this seems rather unlikely. The global spherical shell conjecture \cite{dloutele} would imply that any class-{\sf VII} surface
is  diffeomorphic to a blow-up of  either a Hopf surfaces or an  Inoue-Bombieri surface, and it is moreover  
definitively  known \cite{teleman} that this assertion does  at least hold  for surfaces  with small  $b_2$. 
If the global spherical shell conjecture were true,
Theorem \ref{michael} would thus immediately imply that  any complex surface of 
$\kod = -\infty$ must necessarily have  non-negative Yamabe invariant.

\vfill

\noindent 
{\sf Acknowledgments.} The first author would like to thank   Daniele Angella and Vestislav Apostolov 
for pointing out a number of useful references, while 
the second author would like to thank  Ian Agol, Misha Gromov,  and Peter Kronheimer for
their  extremely  helpful comments and suggestions.

\pagebreak 
%

\vfill
\noindent
{\sc CIRGET,
Universit\'e du Qu\'ebec \`a Montr\'eal, 
Case postale 8888, Succursale centre-ville,
Montr\'eal (Qu\'ebec) H3C 3P8, Canada}

\bigskip 

\noindent
{\sc Department of Mathematics, Stony Brook University (SUNY), Stony Brook, NY 11794-3651, USA}

\end{document}